\documentclass[12pt,draft,leqno]{article}
\usepackage{amssymb,eucal,latexsym}
\usepackage{amsmath,amsthm,latexsym,amscd,amsbsy,fleqn,graphicx}
\usepackage{xy}
\xyoption{all}

\newtheorem{theorem}{Theorem}[section]

\newtheorem{cor}[theorem]{Corollary}
\newtheorem{pro}[theorem]{Proposition}
\newtheorem{defi}[theorem]{Definition}

\newtheorem{nota}[theorem]{Notation}

\newtheorem{fact}[theorem]{Fact}

\newtheorem{nist}[theorem]{}

\newcommand{\tiff}{if and only if \ }
\newcommand{\df}{\ensuremath{\overset{\mathrm{df}}{=}}}

\def\p{\varphi}
\def\a{\alpha}

\def\TE{\Theta}

\def\s{\sigma}

\def\lra{\longrightarrow}

\def\sbe{\subseteq}

\def\stm{\setminus}
\def\ems{\emptyset}
\def\nes{\neq\emptyset}

\def\ex{\exists}
\def\fa{\forall}

\def\we{\wedge}

\def\bv{\bigvee}

\def\ap{^{\,\prime}}
\def\inv{^{-1}}
\def\st{\ |\ }

\def\card #1{\vert #1 \vert}

\def\1{{\bf 1}}
\def\2{\mbox{{\bf 2}}}
\def\3{\mbox{{\bf 3}}}

\def\CC{{\cal C}}

\def\OO{{\cal O}}

\def\TT{{\cal T}}

\def\FFF{{\sf F}}
\def\GGG{{\sf G}}

\def\PPP{{\sf P}}

\def\SSS{{\sf S}}
\def\TTT{{\sf T}}

\def\co{{\sf CO}}

\def\ZHLC{{\bf BooleSp}}

\def\GBPL{{\bf GBoole}}
\def\ZHC{{\bf Stone}}
\def\Bool{{\bf Boole}}

\def\MBool{{\bf mzMaps}}
\def\LMBool{{\bf lmzMaps}}

\def\ZLBA{{\bf ZLBA}}

\def\DZA{{\bf dzBoole}}
\def\LDZA{{\bf ldzBoole}}

\def\int{\mbox{{\rm int}}}

\def\cl{\mbox{{\rm cl}}}

\def\CO{\mbox{{\rm CO}}}
\def\KO{\mbox{{\rm KO}}}

\def\At{{\rm At}}
\def\Att{{\sf At}}

\def\sq{\hspace*{\fill} \hbox{\vrule\vbox{\hrule\phantom{o}\hrule}\vrule}}
\def\sqs{\sq \vspace{2mm}}

\def\Top{{\bf Top}}

\def\ZH{{\bf ZDHaus}}

\def\Caba{{\bf Caba}}

\def\NNNN{\mathbb{N}}

\def\Set{{\bf Set}}

\def\LBA{{\bf LBA}}

\def\tcx{t_X^C}
\def\tcy{t_Y^C}
\def\tcx0{t_{(X,X_0)}}
\def\tcy0{t_{(Y,Y_0)}}

\def\bU0{\bar{U}=(U^0,(U^i,U^{ci})_{i\in\omega})}

\def\bV0{\bar{V}=(V^0,(V^i,V^{ci})_{i\in\omega})}

\title{{\LARGE\bf Extensions of the Stone Duality}\\
\vspace{0.2cm} {\LARGE\bf  to the category $\ZHLC$}\\
\vspace{0.5cm}
{\large\bf G. Dimov and E. Ivanova-Dimova}\thanks{The authors
 were supported by the Bulgarian National Fund of Science, contract no. DN02/15/19.12.2016.}
\\
\vspace{0.2cm} {\footnotesize\rm Faculty of Math. and Informatics,
Sofia University,} {\footnotesize\rm 5 J. Bourchier Blvd., 1164
Sofia, Bulgaria}
}

\author{}

\date{}

\begin{document}

\maketitle

\begin{abstract}
In \cite{DD}, extending the Stone Duality Theorem, we proved two duality theorems for
the category $\ZH$ of
zero-dimensional Hausdorff spaces and continuous maps. Now we derive from them the extension of the Stone Duality Theorem to the category $\ZHLC$
of zero-dimensional locally compact Hausdorff spaces and continuous maps obtained in \cite{D-PMD12}, as well as two new duality theorems for the category $\ZHLC$.
\end{abstract}



\footnotetext[3]{{\footnotesize {\em E-mail addresses:}
gdimov@fmi.uni-sofia.bg, elza@fmi.uni-sofia.bg}}

\section{Introduction}

In 1937, M. Stone \cite{ST}
proved that there exists a bijective correspondence $T_l$ between the
class of all (up to homeomorphism) zero-dimensional locally
compact Hausdorff spaces (briefly, {\em Boolean spaces}\/) and the
class of all (up to isomorphism) generalized Boolean
algebras (briefly, GBAs) (or, equivalently, Boolean rings
with or without unit). In the class of compact Boolean spaces
(briefly, {\em Stone spaces}\/) this bijection can be extended to
a duality $\TTT:\ZHC\lra\Bool$ between the category $\ZHC$ of Stone
spaces and continuous maps and the category $\Bool$ of Boolean
algebras and Boolean homomorphisms; this is the classical Stone Duality Theorem.
In 1964, H. P. Doctor \cite{Do} showed that the Stone bijection $T_l$ can be extended even to a duality between the
category $\ZHLC_{\bf perf}$ of all Boolean spaces and all perfect maps between them and the category $\GBPL$ of all GBAs and suitable morphisms between them.
Later on, G. Dimov \cite{D-a0903-2593,D-PMD12} extended the Stone
Duality to the category $\ZHLC$ of Boolean spaces and continuous maps.
Finally, in \cite{DD}, we extended the Stone
Duality to the category $\ZH$ of zero-dimensional Hausdorff spaces and continuous maps.

In this paper, which can be regarded as a continuation of \cite{DD}, we show how the Dimov Duality Theorem mentioned above can be derived from one of our general duality theorems for the category $\ZH$ proved in \cite{DD}. Moreover, with the help of our results from \cite{DD}, we obtain two new extensions of the Stone Duality Theorem to the category $\ZHLC$, one of which is in the spirit of the recent Duality Theorem of Bezhnanishvili, Morandi and Olberding \cite{BMO} for the category ${\bf Tych}$ of Tychonoff spaces and continuous maps.

The paper is organized as follows. In Section 2 we collect all preliminary facts which are needed for the exposition of our results. In particular, we briefly recall and fix the notation pertaining to all, the Stone Duality, the Tarski Duality and our two duality theorems from \cite{DD} with which we extended the Stone Duality Theorem to the category $\ZH$. In Section 3 we introduce the notion of {\em Boolean ldz-algebra} and present our first new duality theorem for the category $\ZHLC$ (see Theorem \ref{genstonecnew1}). It extends the Stone Duality Theorem and is obtained with the help of our duality theorem \cite[Theorem 3.15]{DD}. After that, using it, we give a new proof of the Dimov Duality Theorem for the category $\ZHLC$ (see \ref{nistboolsp} and Theorem \ref{genstonecnew2}). Finally, in Section 4, we introduce the notion of {\em lmz-map} and with its help we obtain our second new duality theorem for the category $\ZHLC$ (see Theorem \ref{genstonecnew111}). Its proof is based on our duality theorem \cite[Theorem 4.8]{DD} for the category $\ZH$.

 We now fix the general notation.

 Throughout, $(B, \land, \lor, {}^*, 0, 1)$ will denote a Boolean algebra unless indicated otherwise;
we do not assume that $0 \neq 1$. With some abuse of language, we
shall usually identify algebras with their universe, if no
confusion can arise.

 We denote by $\2$ the simplest Boolean
algebra containing only $0$ and $1$, where $0\neq 1$, and by $\NNNN^+$ the set of positive integers.

 If $A$ is a Boolean algebra, then  $A^+ \df A \setminus \{0\}$ and $\At(A)$ is the set of all atoms of $A$.

 If $X$ is a set, we denote by $P(X)$ the power set of $X$;
 clearly, $(P(X),\cup,\cap,\stm,\ems, X)$  $(=(P(X),\sbe))$ is a complete atomic Boolean
 algebra.

 If $X$ is a topological space, we denote by $\CO(X)$ the set of
all clopen (= closed and open) subsets of $X$, and by $\KO(X)$ the set of all compact open subsets of $X$. Obviously,
$(\CO(X),\cup,\cap,\stm,\ems, X)$ $(=(\CO(X),\sbe))$ is a Boolean algebra.
If  $M$ is a subset of $X$, we
denote by  $\cl_X(M)$ (or simply by $\cl(M)$)
the closure of $M$ in $(X,\TT)$.

If $\CC$ is a category, we denote by $\card\CC$ the class of the objects of $\CC$ and by $\CC(X,Y)$ the set of all
   $\CC$-morphisms between two $\CC$-objects $X$ and $Y$.

We denote by:

\begin{itemize}

\item $\Set$ the category of  sets and  functions,

\item  $\Top$ the category of  topological spaces and continuous maps,

\item  $\ZH$ the category of all zero-dimensional Hausdorff spaces and continuous maps,

\item $\ZHC$ the category of all compact Hausdorff zero-dimensional spaces (= {\em Stone spaces}) and their continuous maps,

\item  $\ZHLC$ the category of all locally compact Hausdorff zero-dimensional spaces (= {\em Boolean spaces}) and their continuous maps,

\item $\Bool$ the category of  Boolean algebras and Boolean  homomorphisms,

\item $\Caba$ the category of all complete atomic Boolean algebras and all complete Boolean homomorphisms between them.
\end{itemize}

   The main reference books for all notions which are not defined here
 are \cite{AHS,MacLane,kop89,E}.

 \section{Preliminaries}

We will first recall briefly the Stone Duality Theorem and the Tarski Duality Theorem, and will fix the notation.

\begin{nist}\label{nistone}
\rm We will denote by $\co:\Top\lra\Bool$ the contravariant functor which assigns to every $X\in|\Top|$ the Boolean algebra $(\CO(X),\sbe)$ and to every $f\in\Top(X,Y)$, the Boolean homomorphism $\co(f):\co(Y)\lra\co(X)$ defined by $\co(f)(U)\df f\inv(U)$, for every $U\in\CO(Y)$.

Now we will briefly describe  the Stone duality \cite{ST}
between the categories $\Bool$ and $\ZHC$ using its presentation
given in \cite{H}. We will define two contravariant functors
$$\SSS :\Bool\lra\ZHC \ \ \mbox{ and } \ \ \TTT :\ZHC\lra \Bool.$$ For any Boolean algebra $A$, we let the
space $\SSS (A)$  to be the set
$$X_A\df\Bool(A,\2)$$ endowed with a topology $\TT_A$ having as a
closed base the family $\{s_A(a)\st a\in A\}$, where
\begin{equation}\label{sofa}
s_A(a)\df\{x\in X_A\st x(a)=1\},
\end{equation}
for every $a\in A$; then $\SSS (A)= (X_A,\TT_A)$ is a
 Stone
space. Note that the family $\{s_A(a)\st a\in A\}$ is also an open base of
the space $(X_A,\TT_A)$.

If $\p\in\Bool(A,B)$, then we define $\SSS (\p): \SSS (B)\lra \SSS (A)$
by the formula $\SSS (\p)(y)\df y\circ\p$ for every $y\in \SSS (B)$.
It is easy to see that $\SSS $ is a contravariant functor.

The contravariant functor $\TTT$ is defined to be the restriction of the contravariant functor $\co$ to the category $\ZHC$.

Also, we recall that {\em the Stone map}
\begin{equation}\label{stonemap}
s_A:A\lra \TTT (\SSS (A)), \ \ a\mapsto s_A(a),
\end{equation}
is a $\Bool$-isomorphism.
\end{nist}

\begin{nist}\label{tar}
\rm We will need the Tarski Duality between the categories $\Set$
and $\Caba$. It consists of two contravariant functors
$${\sf P}:\Set\lra\Caba \ \  \mbox{ and } \ \ {\sf At}:\Caba\lra\Set$$
which are defined as follows. For every set $X$, $${\sf P}(X)\df
(P(X),\sbe).$$ If $f\in\Set(X,Y)$, then ${\sf P}(f): {\sf P}(Y)\lra
{\sf P}(X)$ is defined by the formula $${\sf P}(f)(M)\df f\inv(M),$$ for
every $M\in P(Y)$. Further, for every $B\in|\Caba|$,
$${\sf At}(B)\df\At(B);$$ if $\s\in\Caba(B,B\ap)$, then
${\sf At}(\s):{\sf At}(B\ap)\lra {\sf At}(B)$ is defined by the formula
$${\sf At}(\s)(x\ap)\df \bigwedge\{b\in B\st x\ap\le\s(b)\},$$ for every
$x\ap\in \At(B\ap)$.
%
%
%
\end{nist}

Now we will list some simple facts about  some well known constructions of  Boolean homomorphism. The notation which we will introduce will be used  throughout the paper.

\begin{nist}\label{hats}
\rm
Let $\a\in\Bool(A,B)$ and $x\in\At(B)$. Then it is easy to see that the map $$\a_x:A\lra \2$$ defined by
$\a_x(a)=1\Leftrightarrow x\le\a(a)$, for $a\in A$, is a Boolean homomorphism. We put $$X_\a\df\{\a_x\st x\in\At(B)\}.$$
{\em Note that if  $\a$ is a complete Boolean homomorphism, then, for every $x\in \At(B)$, $\a_x$ is a complete Boolean homomorphism as well.}
We put $$h_\a:\At(B)\lra X_\a,\ \ x\mapsto \a_x.$$ It is easy to see that {\em if  every atom of $B$ is a meet of some
elements of $\a(A)$, then $h_\a$ is a bijection.}


If $X$ is a
topological space,  $A=(\CO(X),\sbe)$, $B=\PPP(X)$
and $\a$ is the inclusion map, then, obviously, $X=\At(B)$ and the map $\a_x$ is defined by $\a_x(U)=1\Leftrightarrow x\in U$, for every $U\in A$ and every $x\in X$. In order to simplify the notation, for such $A$ and $B$, we will write $\hat{x}$
instead of $\a_x$.
({\em Note that every $\hat{x}$ is a complete Boolean homomorphism.}) Thus, in such a case, by $$\hat{x}:\co(X)\lra\2$$ we will understand the map defined by $\hat{x}(U)=1\Leftrightarrow x\in U$, for every $U\in \co(X)$; also, we will write
$\hat{X}$ instead of $X_\a$, and
$\hat{h}_X$
instead of $h_\a$, i.e.,
$$\hat{X}=\{\hat{x}:\co(X)\lra\2\st x\in X\}$$
and
$$\hat{h}_{X}:X\lra \hat{X},\ \ x\mapsto \hat{x}.$$
Note that  if the family $\CO(X)$ {\em $T_0$-separates the points of $X$} (i.e., for every $x,y\in X$ such that $x\neq y$,  there exists $U\in \CO(X)$ with $|U\cap\{x,y\}|=1$), then {\em the map
$\hat{h}_{X}$
is a bijection.}
%
\end{nist}

We have to recall some definitions and results from \cite{DD} which will be of great importance for our exposition.

\begin{defi}\label{zalg}{\rm \cite{DD}}
\rm A pair $(A,X)$, where $A\in|\Bool|$
and $X\sbe
\Bool(A,\2)$, is called a {\em Boolean z-algebra} (briefly, {\em
z-algebra}; abbreviated as ZA) if for each $a\in A^+$ there exists
$x\in X$ such that $x(a)=1$.
\end{defi}

\begin{fact}\label{zalgf}{\rm \cite{DD}}
 A pair $(A,X)$ is a z-algebra \tiff $A$ is a Boolean algebra and $X$ is
 a dense subset of\/ $\SSS (A)$.
\end{fact}

\begin{nota}\label{zalgn}
\rm If $A$ is a Boolean algebra and $X\sbe \Bool(A,\2)$,  we set
$$s_A^X(a)\df X\cap s_A(a)$$ for each $a\in A$, defining in such a way a map $$s_A^X: A\lra \PPP(X), \ \
a\mapsto s_A^X(a).$$
\end{nota}


\begin{defi}\label{dzalgn}{\rm \cite{DD}}
\rm A z-algebra $(A,X)$ is called a {\em Boolean dz-algebra}
(briefly, {\em dz-algebra}; abbreviated as DZA) if $s_A^X(A)=\CO(X)$ (where $X$ is regarded as a subspace of $\SSS(A)$).
\end{defi}

\begin{pro}\label{zboolpro}{\rm \cite{DD}}
There is a category $\DZA$  whose objects are all dz-algebras  and
whose morphisms between any two $\DZA$-objects   $(A,X)$ and
$(A\ap,X\ap)$ are all pairs $(\p,f)$ such that
$\p\in\Bool(A,A\ap)$, $f\in\Set(X\ap,X)$ and $x\ap\circ\p=f(x\ap)$
for every $x\ap\in X\ap$. The composition $(\p\ap,f\ap)\circ (\p,f)$ between two $\DZA$-morphisms
$(\p,f):(A,X)\lra(A\ap,X\ap)$ and
$(\p\ap,f\ap):(A\ap,X\ap)\lra (A'',X'')$ is defined to be the $\DZA$-morphism
$(\p\ap\circ\p, f\circ f\ap):(A,X)\lra (A'',X'')$; the identity
morphism of a $\DZA\ap$-object  $(A,X)$ is defined to be $(id_A,id_X)$.
\end{pro}

\begin{theorem}\label{zduality}{\rm \cite{DD}}
The categories\/ $\ZH$ and\/ $\DZA$ are dually equivalent.
\end{theorem}

\noindent{\em Sketch of the proof.} For every  $X\in|\ZH|$, we let
$$F(X)\df(\co (X),\hat{X}).$$
Then $F(X)\in|\DZA|$. For  $f\in\ZH(X,Y)$, set
 $$F(f)\df (\co (f),\hat{f}),$$
 where
$$\hat{f}:\hat{X}\lra \hat{Y}$$ is defined by
$$\hat{f}(\hat{x})\df\widehat{f(x)}$$ for every $x\in X$. Then
$F(f)\in\DZA(F(Y),F(X))$ and $$F:\ZH\lra\DZA$$ is a contravariant functor.

For every $(A,X)\in|\DZA|$, we set $$G(A,X)\df X,$$ where $X$ is regarded as a subspace of
$\SSS (A)$. Then $G(A,X)\in|\ZH|$.

If
$(\p,f):(A,X)\lra (A\ap,X\ap)$
is a $\DZA$-morphism, we put
$$G(\p,f)\df f.$$
Then $G(\p,f)$ is a continuous function and $$G:\DZA\lra\ZH$$ is
a contravariant functor. Moreover, the functors $F\circ G$ and $G\circ F$ are
naturally isomorphic to the corresponding identity functors.
\sqs

\begin{defi}\label{zhomo}{\rm \cite{DD}}
\rm Let $A$ be a Boolean algebra and $B\in|\Caba|$. A Boolean
monomorphism $\a:A\lra B$ is said to be a {\em Boolean z-map}
(briefly, {\em z-map}) if every atom of $B$ is a meet of some
elements of $\a(A)$. A z-map $\a:A\lra B$ is called a {\em maximal
Boolean z-map} (briefly, {\em mz-map}) if $\CO(X_\a)=s_A^{X_\a}(A)$, where $X_\a$ is regarded as a subspace of $\SSS(A)$
(see \ref{hats} and \ref{zalgn} for the notation).
\end{defi}

\begin{pro}\label{mboolpro}{\rm \cite{DD}}
There is a category $\MBool$  whose objects are all mz-maps  and
whose morphisms between any two $\MBool$-objects  $\a:A\lra B$ and
$\a\ap:A\ap\lra B\ap$ are all pairs $(\p,\s)$ such that
 $\p\in\Bool(A,A\ap)$, $\s\in\Caba(B,B\ap)$ and
$\a\ap\circ\p=\s\circ\a$, i.e., the diagram
\begin{center}
$\xymatrix{A\ar[rr]^{\p}\ar[d]_{\a} && A\ap\ar[d]^{\a\ap}\\
B\ar[rr]^{\s} && B\ap
}$
\end{center}
is commutative.
The composition $(\p\ap,\s\ap)\circ (\p,\s)$ between two $\MBool$-morphisms
$(\p,\s):\a\lra\a\ap$ and
$(\p\ap,\s\ap):\a\ap\lra \a''$ is defined to be the $\MBool$-morphism
$(\p\ap\circ\p, \s\ap\circ \s):\a\lra \a''$; the identity map of
an $\MBool$-object  $\a:A\lra B$ is defined to be $(id_A,id_B)$.
\end{pro}

\begin{theorem}\label{nzduality}{\rm \cite{DD}}
The categories\/ $\ZH$ and\/ $\MBool$ are dually equivalent.
\end{theorem}

\noindent{\em Sketch of the proof.}
We will  define two contravariant functors
$$\FFF:\ZH\lra \MBool\ \ \mbox{ and  }\ \ \GGG:\MBool\lra\ZH.$$

For every $X\in|\ZH|$, we put $$\FFF(X)\df i_X,$$ where $i_X:\co(X)\lra\PPP(X)$ is the inclusion map.
Then $i_X\in|\MBool|$. For $f\in\ZH(X,Y)$, we set $$\FFF(f)\df (\co(f),\PPP(f)).$$
Then $\FFF(f)$ is a $\MBool$-morphism.

For $(\a:A\lra B)\in|\MBool|$, we put $$\GGG(\a)\df X_\a.$$
Clearly, the set $X_\a$ endowed with the subspace topology from the space $\SSS(A)$ is a $\ZH$-object.
For $(\p,\s)\in\MBool(\a,\a\ap)$, with $\a:A\lra B$ and $\a\ap:A\ap\lra B\ap$, we set $$\GGG(\p,\s)\df f_\s,$$
where $f_\s:X_{\a\ap}\lra X_\a$ is defined by the formula $f_\s(\a\ap_{x\ap})=a_{\Att(\s)(x\ap)}$, for every $x\ap\in\At(B\ap)$.
Then $\GGG(\p,\s)$ is a $\ZH$-morphism.

We obtain that $\FFF$ and $\GGG$ are contravariant functors. Their compositions are naturally isomorphic to the corresponding identity functors.
\sqs

We have to recall as well some results from \cite{D-PMD12}  concerning the Dimov extension of the Stone Duality Theorem to the category $\ZHLC$.

\begin{nist}\label{p1}
\rm Recall that if $(A,\le)$ is a poset and $B\sbe A$ then $B$ is
said to be a {\em dense subset of} $A$ if for any $a\in
A^+$ there exists $b\in B^+$ such that $b\le a$;
when $(B,\le_1)$ is a poset and $f:A\lra B$ is a map, then we will
say that $f$ is a {\em dense map}\/ if $f(A)$ is a dense subset of
$B$.

Recall that a {\em frame} is a complete lattice $L$ satisfying
the infinite distributive law $a\we\bigvee S=\bigvee\{a\we s\st
s\in S\}$, for every $a\in L$ and every $S\sbe L$.

Let $A$ be a distributive $\{0\}$-pseudolattice and $Idl(A)$ be
the frame of all ideals of $A$. If $J\in Idl(A)$ then we will
write $\neg_A J$ (or simply $\neg J$) for the pseudocomplement of
$J$ in $Idl(A)$ (i.e.,  $\neg J=\bv\{I\in Idl(A)\st I\we
J=\{0\}\}$). Note that $\neg J=\{a\in A\st (\fa b\in J)(a\we
b=0)\}$ (see Stone \cite{ST}). Recall that an ideal $J$ of $A$ is
called {\em simple} (Stone \cite{ST}) if $J\vee\neg J= A$ (i.e.,  $J$ is a complemented element of the frame $Idl(A)$). As it
is proved in \cite{ST}, the set $Si(A)$ of all simple ideals of
$A$ is a Boolean algebra with respect to the lattice operations in
$Idl(A)$.
\end{nist}

\begin{defi}\label{deflba}{\rm \cite{D-PMD12}}
\rm A pair $(A,I)$, where $A$ is a Boolean algebra and $I$ is an
ideal of $A$ (possibly non proper) which is dense in $A$ (shortly,
dense ideal), is called a {\em local Boolean algebra} (abbreviated
as LBA).

Let $\LBA$ be the category whose objects are all LBAs and whose
morphisms are all functions $\p:(A, I)\lra(B, J)$ between the
objects of $\LBA$ such that $\p:A\lra B$ is a Boolean homomorphism
satisfying the following condition:

\smallskip

\noindent(LBA) For every $b\in J$ there exists $a\in I$ such that
$b\le \p(a)$;

\smallskip

\noindent let the composition between the morphisms of $\LBA$ be
the usual composition between functions, and the $\LBA$-identities
be the identity functions.
\end{defi}

Note that two LBAs $(A,I)$ and $(B,J)$ are
$\LBA$-isomorphic  if there exists a
Boolean isomorphism $\p:A\lra B$ such that $\p(I)=J$.

\begin{defi}\label{defzlba}{\rm \cite{D-PMD12}}
\rm An LBA $(A, I)$  is called a {\em ZLB-algebra} (briefly, {\em
ZLBA}) if, for every $J\in Si(I)$, the join $\bv_A J$($=\bv_A
\{a\st a\in J\}$) exists.
\end{defi}

\begin{nota}\label{notalia}
For every LBA $(A, I)$, we put
$$L_I^A=\{x\in \Bool(A,\2)\st 1\in x(I)\}.$$
\end{nota}

\begin{pro}\label{cox}{\rm \cite{D-PMD12}}
Let $(A,I)$ be an LBA.  Then $(A,I)$ is a ZLBA
iff $L_I^A\cap s_A(A)=\CO(L_I^A)$ (where $L_I^A$ is regarded as a subspace of $\SSS(A)$).
\end{pro}

Let $\ZLBA$ be the full subcategory of the category $\LBA$ having
as objects  all ZLBAs.

\begin{theorem}\label{genstonecnew}{\rm \cite{D-PMD12}}
The categories\/ $\ZHLC$ and\/ $\ZLBA$ are  dually equivalent.
\end{theorem}

\noindent{\em Sketch of the proof.}
We will define two contravariant functors
$$\TE^a_d:\ZLBA\lra\ZHLC\ \ \mbox{ and }\ \ \TE^t_d:\ZHLC\lra\ZLBA$$
 as follows. For $X\in\card{\ZHLC}$, we set $$\TE^t_d(X)\df (\co(X),  \KO(X)).$$
Then $\TE^t_d(X)$ is a ZLBA. For every $f\in\ZHLC(X,Y)$, we put $$\TE^t_d(f)\df \co(f).$$
Then $\TE^t(f)$ is a $\ZLBA$-morphism.

For every ZLBA $(A, I)$, we set
$$\TE^a_d(A, I)\df L_I^A.$$
Then $\TE^a_d(A, I)\in|\ZHLC|$.

If $\p\in\ZLBA((A, I),(B, J))$, then we define the map
$\TE^a_d(\p):\TE^a_d(B, J)\lra\TE^a_d(A, I)$ by the formula
$$\TE^a_d(\p)(x\ap)\df x\ap\circ\p, \ \  \fa x\ap\in\TE^a_d(B,J).$$
Then $\TE^a_d(\p)$ is a $\ZHLC$-morphism.

Finally, we show that the compositions $\TE^a_d\circ \TE^t_d$ and $\TE^t_d\circ \TE^a_d$ are naturally  isomorphic to the corresponding identity functors.
\sqs

We will need as well the following theorem of M. Stone \cite{Stone} (see also \cite[Theorem 7.25]{kop89}):

\begin{theorem}\label{opensetidlstd}
Let $A\in|\Bool|$ and $(X_A,\TT_A)=\SSS (A)$.
Then there exists an order-preserving bijection (and hence, frame isomorphism)
$$\iota:(Idl(A),\le)\lra (\TT_A,\sbe), \ \ J\mapsto
\bigcup\{s_A(a)\st a\in J\}.$$ If $U\in\TT_A$ then $$J=
\iota\inv(U)=\{a\in A\st s_A(a)\sbe U\}.$$
\end{theorem}

\section{A new extension of the Stone Duality Theorem to the category $\ZHLC$ and a new proof of Theorem \ref{genstonecnew}}

In this section we will introduce the notion of {\em Boolean ldz-algebra} and with its help we will obtain our first new duality theorem for the category $\ZHLC$, namely, Theorem \ref{genstonecnew1}. Using it, we will present a new proof of the Dimov Duality Theorem \cite{D-PMD12} for the category $\ZHLC$ (cited here as Theorem \ref{genstonecnew}).

\begin{defi}\label{ldzalgn}
\rm A dz-algebra $(A,X)$ is called a {\em Boolean ldz-algebra}
(briefly, {\em ldz-algebra}; abbreviated as LDZA) if for every $x\in X$ there exists $a\in A$ such that $x\in s_A(a)\sbe X$.
\end{defi}

The following assertion is obvious.

\begin{fact}\label{ldzalgf}
 A dz-algebra $(A,X)$ is an ldz-algebra \tiff $X$ is
 an open subset of\/ $\SSS (A)$.
\end{fact}

Let $\LDZA$ be the full subcategory of the category $\DZA$ having as objects all ldz-algebras.

\begin{theorem}\label{genstonecnew1}
The categories\/ $\ZHLC$ and\/ $\LDZA$ are  dually equivalent.
\end{theorem}

\begin{proof}
Let $X\in|\ZHLC|$. Then, by Theorem \ref{zduality}, $X$ is homeomorphic to the space $G(F(X))=\hat{X}$. Therefore, $\hat{X}$ is a locally compact dense subset of the space $\SSS(\co(X))$. Then \cite[Theorems 3.3.9 and 3.3.8]{E} imply that $\hat{X}$  is an open subset of $\SSS(\co(X))$. Hence, $F(X)=(\co(X),\hat{X})$ is an ldz-algebra. Thus, $F(\ZHLC)\sbe\LDZA$.

Let $(A,X)$ be an ldz-algebra. Then $X$ is an open subspace of the Stone space $\SSS(A)$. Thus, $G(A,X)=X$ is a Boolean space, i.e., $G(\LDZA)\sbe \ZHLC$.
Now, using Theorem \ref{zduality}, we complete the proof of our assertion.
\end{proof}

Now it becomes clear that for proving Theorem \ref{genstonecnew} it is enough to show that the categories\/ $\LDZA$ and\/ $\ZLBA$ are isomorphic.

\begin{theorem}\label{genstonecnew2}
The categories\/ $\LDZA$ and\/ $\ZLBA$ are isomorphic.
\end{theorem}

\begin{proof}
We will define two functors $$E:\LDZA\lra\ZLBA\ \ \mbox{ and }\ \ E\ap:\ZLBA\lra\LDZA,$$ and will show that their compositions are equal to the corresponding identity functors.

Let $(A,X)\in|\LDZA|$. Then $X$ is an open subset of the space $\SSS(A)$. Set
$$I_X\df\{a\in A\st s_A(a)\sbe X\}.$$
We will show that $(A,I_X)\in|\ZLBA|$. Indeed, by Theorem \ref{opensetidlstd}, $I_X$ is an ideal of $A$ (proper or non proper). We have to prove that $I_X$ is a dense subset of $A$. Let $b\in A^+$. Then $s_A(b)\nes$. By Fact \ref{zalgf}, $X$ is a dense subset of $\SSS(A)$. Hence $X\cap s_A(b)$ is an open non-empty subset of $\SSS(A)$. Then there exists $a\in A^+$ such that $s_A(a)\sbe X\cap s_A(b)$. Thus $a\in I_X$ and $a\le b$. So, $I_X$ is a dense ideal of $A$ and, hence, $(A,I_X)$ is an LBA. For proving that it is a ZLBA, we will use Proposition \ref{cox} (and its notation). We will first show that $L_{I_X}^A=X$. Indeed, let $x\in X$. Then there exists $a\in A$ such that $x\in s_A(a)\sbe X$. Hence $a\in I_X$ and $x(a)=1$. Thus, $X\sbe  L_{I_X}^A$. Conversely, let $x\in L_{I_X}^A$. Then there exists $a\in I_X$ such that $x(a)=1$. Thus $x\in s_A(a)\sbe X$. Therefore, $x\in X$. So, $$L_{I_X}^A=X.$$ Since $(A,X)$ is a DZA, we have that $X\cap s_A(A)=\CO(X)$. Hence, $L_{I_X}^A\cap s_A(A)=\CO(L_{I_X}^A)$. Then Proposition \ref{cox} shows that $(A,I_X)$ is a ZLBA.
We set now: $$E(A,X)\df (A,I_X).$$

Further, let $(\p,f)\in\LDZA((A,X),(B,X\ap))$. Then $\p\in\Bool(A,B)$, $f\in\Set(X\ap,X)$ and $f(x\ap)=x\ap\circ\p$, for every $x\ap\in X\ap$. We will show that
$\p:(A,I_X)\lra (B,I_{X\ap})$ is a $\ZLBA$-morphism.  Set, for short, $I\df I_X$ and $J\df I_{X\ap}$. We need only to prove that $\p$ satisfies condition (LBA), i.e., we have to show that for every $b\in J$ there exists $a\in I$ such that $b\le\p(a)$. Let $b\in J$. Then $s_B(b)\sbe X\ap$ and thus, $f(s_B(b))\sbe X$. Set
$\Phi\df\uparrow\!\!(b)(=\{c\in B\st b\le c\})$. Then $\Phi$ is a filter in $B$. Suppose that $\Phi\cap \p(I)=\ems$. Let $I_0$ be the ideal in $B$ generated by $\p(I)$. We will prove that $\Phi\cap I_0=\ems$. Let $c\in I_0$. Then $c\le b_1\vee\ldots\vee b_n$ for some $n\in \NNNN^+$ and some $b_1,\ldots,b_n\in\p(I)$. For every $i=1,\ldots,n$, there exists $a_i\in I$ such that $b_i=\p(a_i)$. Setting $a_0\df\bigvee_{i=1}^n a_i$, we obtain that $a_0\in I$ and $b_1\vee\ldots\vee b_n=\p(a_0)$. Thus, $c\le\p(a_0)$. Suppose that $c\in\Phi$. Then $\p(a_0)\in\Phi$, a contradiction. Therefore, $\Phi\cap I_0=\ems$. Then, by the Maximal Ideal Theorem \cite{ST} (see also \cite[Lemma 2.3]{J}), there exists an ideal $I\ap$ of $B$ which is maximal amongst those containing $I_0$
and disjoint from $\Phi$. Clearly, $I\ap$ is an ideal of $B$ which is maximal amongst those disjoint from $\Phi$. Then, as it is well known (see, e.g., \cite[Theorem 2.4]{J}), $I\ap$ is a prime ideal of $B$. Thus there exists a homomorphism $x\ap\in\Bool(B,\2)$ such that $I\ap=(x\ap)\inv(0)$ (see, e.g., \cite[Proposition 2.2]{J}). Hence $x\ap(b)=1$, i.e., $x\ap\in s_B(b)$ and therefore $x\ap\in X\ap$. Then $f(x\ap)=x\ap\circ\p\in X$. Since $X$ is an open subset of $\SSS(A)$, there exists $a\in A$ such that $x\ap\circ\p\in s_A(a)\sbe X$. Then $a\in I$ and thus $\p(a)\in\p(I)\sbe I_0\sbe I\ap$. Hence, $x\ap(\p(a))=0$. Since $x\ap\circ\p\in s_A(a)$, we obtain a contradiction. Therefore, $\p(I)\cap \Phi\nes$. This means that there exists $a\in I$ such that $\p(a)\ge b$. So, $\p$ is a $\ZLBA$-morphism. Now we set
$$E(\p,f)\df \p.$$
It is easy to see that $E$ is a functor.

Let now $(A,I)\in|\ZLBA|$. We set $X_I\df\bigcup\{s_A(a)\st a\in I\}$. Then, clearly, $X_I$ is an open subset of $\SSS(A)$. We will show that $(A,X_I)\in|\LDZA|$. Let $b\in A^+$. Since $I$ is dense in $A$, there exists $a\in I^+$ such that $a\le b$. Then $\ems\neq s_A(a)\sbe X_I\cap s_A(b)$. Thus $X_I$ is a dense subset of $\SSS(A)$. Now, Fact \ref{zalgf} implies that $(A,X_I)$ is a z-algebra. For proving that it is a dz-algebra, we will first show that $X_I=L_I^A$ (see Notation \ref{notalia}). Let $x\in X_I$. Then there exists $a\in I$ such that $x\in s_A(a)$. Thus $x(a)=1$ and, hence, $x\in L_I^A$. So, $X_I\sbe L_I^A$.
Conversely,  if $x\in L_I^A$, then there exists $a\in I$ with $x(a)=1$. Thus $x\in s_A(a)$ and, hence, $x\in X_I$. Therefore, $$X_I=L_I^A.$$ Since, by Proposition \ref{cox}, we have that $L_I^A\cap s_A(A)=\CO(L_I^A)$, we obtain that $X_I\cap s_A(A)=\CO(X_I)$. So, $(A,X_I)\in|\LDZA|$. We set
$$E\ap(A,I)\df (A,X_I).$$

Let $\p\in\ZLBA((A,I),(B,J))$ and $x\ap\in X_J$. Then there exists $b\in J$ with $x\ap(b)=1$. Since $\p$ satisfies condition (LBA), there exists $a\in I$ such that $b\le\p(a)$. Then $1=x\ap(b)\le x\ap(\p(a))$. Hence, $x\ap(\p(a))=1$. This means that $x\ap\circ \p\in X_I$. So, the function $$f_\p:X_J\lra X_I, \ \ x\ap\mapsto x\ap\circ \p,$$ is well-defined. Thus $(\p,f_\p)\in\LDZA(E\ap(A,I),E\ap(B,J))$ and we set
$$E\ap(\p)\df(\p,f_\p).$$
It is easy to see that $E\ap$ is a functor.

We will show that $E\ap\circ E=Id_{\LDZA}$. Let $(A,X)\in|\LDZA|$. Then $E(A,X)=(A,I_X)$, where $I_X=\{a\in A\st s_A(a)\sbe X\}$, and $E\ap(E(A,X))=(A,X_{I_X})$, where $X_{I_X}=\bigcup\{s_A(a)\st a\in I_X\}$. Clearly, $X_{I_X}\sbe X$. Conversely, let $x\in X$. Then there exists $a\in A$ such that $x\in s_A(a)\sbe X$. Thus $a\in I_X$ and, hence, $s_A(a)\sbe X_{I_X}$. This shows that $x\in X_{I_X}$. Therefore, $$X=X_{I_X}.$$ So, $E\ap(E(A,X))=(A,X)$. Further, let $(\p,f)\in\LDZA((A,X),(B,X\ap))$. Then $E\ap(E(\p,f))=E\ap(\p)=(\p,f_\p)$, where $f\in\Set(X\ap,X)$, $f(x\ap)=x\ap\circ\p$, for every $x\ap\in X\ap$, $f_\p\in\Set(X_{I_{X\ap}},X_{I_X})$ and $f_\p(x\ap)=x\ap\circ\p$, for every $x\ap\in X_{I_{X\ap}}$. Since $X_{I_{X\ap}}=X\ap$ and $X_{I_X}=X$ (as we have shown above), we obtain that $f=f_\p$. Therefore, $E\ap\circ E=Id_{\LDZA}$.

Finally, we will show that $E\circ E\ap=Id_{\ZLBA}$. Let  $(A,I)\in|\ZLBA|$. Then $E(E\ap(A,I))=E(A,X_I)=(A,I_{X_I})$, where $X_I=\bigcup\{s_A(a)\st a\in I\}$ and $I_{X_I}=\{a\in A\st s_A(a)\sbe X_I\}$. If $a\in I$, then $s_A(a)\sbe X_I$ and thus $a\in I_{X_I}$. Hence, $I\sbe I_{X_I}$. For proving that $I_{X_I}\sbe I$, let $a\in I_{X_I}$. Then $s_A(a)\sbe X_I$. We  have already proved that $X_I=L_I^A=\{x\in\Bool(A,\2)\st 1\in x(I)\}$. Thus, for every $x\in s_A(a)$ there exists $b_x\in I$ such that $x\in s_A(b_x)$. Therefore, $\{s_A(b_x)\cap s_A(a)\st x\in s_A(a)\}$ is an open cover of the compact space $s_A(a)$. Hence, there exist $n\in\NNNN^+$ and $x_1,\ldots,x_n\in s_A(a)$ such that $s_A(a)=s_A(a)\cap\bigcup_{i=1}^n s_A(b_{x_i})=s_A(a\we\bigvee_{i=1}^n b_{x_i})$. Setting $b\df \bigvee_{i=1}^n b_{x_i}$, we obtain that $b\in I$ and $s_A(a)=s_A(a\we b)$. Thus $a=a\we b$, i.e., $a\le b$. Since $b\in I$, we conclude that $a\in I$. So, $$I_{X_I}=I.$$ Thus $E(E\ap(A,I))=(A,I)$. Let now $\p\in\ZLBA((A,I),(B,J))$. Then  $E(E\ap(\p))=E(\p,f_\p)=\p$. Hence, $E\circ E\ap=Id_{\ZLBA}$.

All this shows that the categories\/ $\LDZA$ and\/ $\ZLBA$ are isomorphic.
\end{proof}

\begin{nist}\label{nistboolsp}
\rm Clearly, Theorems \ref{genstonecnew2} and \ref{genstonecnew1} imply Theorem \ref{genstonecnew}. We will show that even the equalities  $E\circ F_l=\TE_d^t$ and $G_l\circ E\ap=\TE_d^a$ take place, where $F_l:\ZHLC\lra \LDZA$ and $G_l:\LDZA\lra\ZHLC$ are the restrictions of the contravariant functors $F$ and $G$ from the proof of Theorem \ref{zduality}, respectively. Indeed, for every $(A,I)\in |\ZLBA|$, we have $G_l(E\ap(A,I))=G_l(A,X_I)=X_I$ and $\TE_d^a(A,I)=L_I^A$. Since $X_I=L_I^A$ (see the proof of Theorem \ref{genstonecnew2}), we obtain that $G_l(E\ap(A,I))=\TE_d^a(A,I)$. If $\p\in\ZLBA((A, I),(B, J))$, then $G_l(E\ap(\p))=G_l(\p,f_\p)=f_\p$, where
$f_\p:X_J\lra X_I, \ \ x\ap\mapsto x\ap\circ \p$. Also, $\TE_d^a(\p):\TE_d^a(B,J)\lra \TE_d^a(A,I)$ is defined by $\TE_d^a(\p)(x\ap)=x\ap\circ\p$ for every $x\ap\in L_J^B=X_J$. Therefore, $G_l(E\ap(\p))=\TE_d^a(\p)$. So, $G_l\circ E\ap=\TE_d^a$. Further, for every $X\in|\ZHLC|$, $E(F_l(X))=E(\co(X),\hat{X})=(\co(X),I_{\hat{X}})$ and $\TE_d^t(X)=(\co(X),\KO(X))$. We have that $\hat{X}=\{\hat{x}:\co(X)\lra\2\st x\in X\}$, where $\hat{x}(U)=1\Leftrightarrow x\in U$, for every $x\in X$ and every $U\in\co(X)$, and $I_{\hat{X}}=\{U\in\co(X)\st s_{\co(X)}(U)\sbe \hat{X}\}$. Also, $\hat{X}\cap s_{\co(X)}(U)=\{\hat{x}\in\hat{X}\st \hat{x}(U)=1\}=\{\hat{x}\in\hat{X}\st x\in U\}=\hat{U}$.  As it is shown in \cite[Example 3.9]{DD}, the map $\hat{h}_X: X\lra\hat{X},\ \ x\mapsto \hat{x}$, is a homeomorphism (regarding $\hat{X}$ as a subspace of $\SSS(\co(X))$). Therefore, $\hat{h}_X(\CO(X))=\CO(\hat{X})$ and for every $U\in\CO(X)$, $U$ is homeomorphic to $\hat{h}(U)=\hat{U}$. Thus, if $U\in I_{\hat{X}}$ then $\hat{U}=s_{\co(X)}(U)$ and, hence, $\hat{U}$ is compact, so that $U\in\KO(X)$. Therefore, $I_{\hat{X}}\sbe\KO(X)$. Conversely,
if $U\in\KO(X)$ then $\hat{U}$ is compact, i.e., $\hat{X}\cap s_{\co(X)}(U)$ is compact. Since $(\co(X),\hat{X})$ is a ZA (by \cite[Example 3.9]{DD}), $\hat{X}$ is a dense subset of $\SSS(\co(X))$. Thus, setting $Y\df \SSS(\co(X))$, we obtain that $\hat{U}=\cl_Y(\hat{U})=\cl_Y(\hat{X}\cap s_{\co(X)}(U))=\cl_Y(s_{\co(X)}(U))=s_{\co(X)}(U)$. Therefore, $s_{\co(X)}(U)\sbe \hat{X}$, i.e., $U\in I_{\hat{X}}$. Hence, $I_{\hat{X}}=\KO(X)$. Thus $(E\circ F_l)(X)=\TE_d^t(X)$. Finally, if $f\in\ZHLC(X,Z)$, then $E(F_l(f))=E(\co(f),\hat{f})=\co(f)=\TE_d^t(f)$. This shows that $E\circ F_l=\TE_d^t$.
\end{nist}

\section{One more extension of the Stone Duality Theorem to the category $\ZHLC$}

In this section we will introduce the notion of {\em lmz-map} and with its help we will obtain our second new duality theorem for the category $\ZHLC$, namely, Theorem \ref{genstonecnew111}.
For doing this, we have first to recall
 the following theorem from \cite{DD}:

\begin{theorem}\label{zeq}{\rm \cite{DD}}
The categories\/ $\MBool$ and\/ $\DZA$ are equivalent.
\end{theorem}

\noindent{\em Sketch of the proof.} We will first define two functors $F\ap:\DZA\lra\MBool$ and $G\ap:\MBool\lra\DZA$.

For every $(A,X)\in|\DZA|$, we set $$F\ap(A,X)\df s_A^X.$$ Then $s_A^X\in|\MBool|$.
If $(\p,f)\in\DZA((A,X),(A\ap,X\ap))$, we put $$F\ap(\p,f)\df
(\p,\PPP(f)).$$ Then $F\ap(\p,f)\in\MBool(F\ap(A,X),F\ap(A\ap,X\ap))$.

For every $(\a:A\lra B)\in |\MBool|$, we set
$$G\ap(\a)\df(\co (X_\a),\widehat{X_\a}),$$
where $X_\a$ is regarded as a subspace of $\SSS(A)$, $\widehat{X_\a}=\{\widehat{\a_x}:\co (X_\a)\lra\2\st \a_x\in X_\a\}$ and $\widehat{\a_x}(U)=1\Leftrightarrow \a_x\in U$, for every $U\in\co(X_\a)$. Then $G\ap(\a)\in |\DZA|$.
Let $(\p,\s)\in\MBool(\a,\a\ap)$, where $\a:A\lra B$ and $\a\ap:A\ap\lra B\ap$. We set
$$G\ap(\p,\s)\df (\co(f_\s),\widehat{f_\s}),$$
where $f_\s:X_{\a\ap}\lra X_\a$ is defined by $\a\ap_{x\ap}\mapsto \a_{\Att(\s)(x\ap)}$ and $\widehat{f_\s}:\widehat{X_{\a\ap}}\lra\widehat{X_\a}$ is defined by $\widehat{\a\ap_{x\ap}}\mapsto\widehat{f_\s(\a\ap_{x\ap})}$. Then $G\ap(\p,\s)\in \DZA(G\ap(\a),G\ap(\a\ap))$.

It is easy to see
that $F\ap$ and $G\ap$  are functors. Finally, the functors $F\ap\circ G\ap$ and $G\ap\circ F\ap$ are
naturally isomorphic to the corresponding identity functors.
\sqs

\begin{defi}\label{lmzhomo}
\rm  An mz-map $\a:A\lra B$ is called an  {\em lmz-map}\/ if for every $x\in\At(B)$ there exists $a\in A$ such that $\a_x\in s_A(a)\sbe X_\a$.
\end{defi}

\begin{fact}\label{lmzhomof}
An mz-map $\a:A\lra B$ is an lmz-map iff $X_\a$ is an open subset of $\SSS(A)$.
\end{fact}

Let $\LMBool$ be the full subcategory of the category $\MBool$ having as objects all lmz-maps.

\begin{theorem}\label{genstonecnew11}
The categories\/ $\LDZA$ and\/ $\LMBool$ are equivalent.
\end{theorem}

\begin{proof} We will show that the restriction $F\ap_l$ of the functor $F\ap$ (defined in the proof of Theorem \ref{zeq}) to the category $\LDZA$  is in fact a functor between the categories\/  $\LDZA$ and\/ $\LMBool$.
Let $(A,X)\in|\LDZA|$. Then $F\ap(A,X)= s_A^X$. Set $\a\df s_A^X$. Then $\a:A\lra \PPP(X)$ and, as it is easy to show (see also \cite[page 17]{DD}), $X_\a\equiv X$. Since, by the definition of an ldz-algebra, $X$ is an open subset of $\SSS(A)$, we obtain that $X_\a$ is an open subset of $\SSS(A)$. Thus, $F\ap(A,X)\in|\LMBool|$.

Now we will show that the restriction $G\ap_l$ of the functor $G\ap$ (defined in the proof of Theorem \ref{zeq}) to the category $\LMBool$  is in fact a functor between the categories\/ $\LMBool$ and\/ $\LDZA$.
Let $(\a:A\lra B)\in |\LMBool|$. Then
$G\ap(\a)=(\co (X_\a),\widehat{X_\a})$. Set $A\ap\df\co(X_\a)$. We need only to prove that $\widehat{X_\a}$ is an open subset of $\SSS(A\ap)$. We have that $X_\a$ is an open subset of $\SSS(A)$. As it is shown in \cite[page 16]{DD}, the restriction $\bar{s}_A^{X_\a}:A\lra\co (X_\a)=A\ap$ of the map $s_A^{X_\a}:A\lra\PPP(X_\a)$ is a Boolean isomorphism. Then, by the Stone Duality Theorem, the map $f\df\SSS(\bar{s}_A^{X_\a}):\SSS(A\ap)\lra\SSS(A)$ is a homeomorphism. We will prove that $f(\widehat{X_\a})=X_\a$. Clearly, this will imply that $\widehat{X_\a}$ is an open subset of $\SSS(A\ap)$. Let $\widehat{\a_x}\in\widehat{X_\a}$. Then $x\in\At(B)$ and $f(\widehat{\a_x})=\widehat{\a_x}\circ\bar{s}_A^{X_\a}$. Thus, for every $a\in A$, $f(\widehat{\a_x})(a)=1\Leftrightarrow \widehat{\a_x}(\bar{s}_A^{X_\a}(a))=1\Leftrightarrow \a_x\in \bar{s}_A^{X_\a}(a)\Leftrightarrow \a_x(a)=1$. Hence, $f(\widehat{\a_x})=\a_x$. This shows that $f(\widehat{X_\a})=X_\a$. Therefore, $G\ap(\a)\in|\LDZA|$.

So, $F\ap_l:\LDZA\lra\LMBool$ and $G\ap_l:\LMBool\lra\LDZA$  are restrictions of the functors $F\ap:\DZA\lra\MBool$ and $G\ap:\MBool\lra\DZA$, respectively. Now, our assertion follows from Theorem \ref{zeq}.
\end{proof}

\begin{theorem}\label{genstonecnew111}
The categories\/ $\ZHLC$ and\/ $\LMBool$ are dually equivalent.
\end{theorem}

\begin{proof}
Obviously, our assertion follows immediately from Theorems \ref{genstonecnew11} and \ref{genstonecnew1}.
In the rest of this proof we will describe explicitly the contravariant functors between the categories\/ $\ZHLC$ and\/ $\LMBool$ which realize the requested duality and, moreover, we will simplify them.

Set $\FFF_0^{\, l}\df F\ap_l\circ F_l$ and $\GGG_0^{\, l}\df G_l\circ G\ap_l$. Then $\FFF_0^{\, l}:\ZHLC\lra \LMBool$ and $\GGG_0^{\, l}:\LMBool\lra \ZHLC$. Clearly, these contravariant functors realize the duality between the categories\/ $\ZHLC$ and\/ $\LMBool$ and they are restrictions of the functors $\FFF_0\df F\ap\circ F:\ZH\lra \MBool$ and $\GGG_0\df G\circ G\ap:\MBool\lra \ZH$, respectively. As it is shown in \cite[page 20]{DD}, the contravariant functors $\FFF_0$ and $\FFF$, as well as $\GGG_0$ and $\GGG$, are naturally isomorphic (see the proof of Theorem \ref{nzduality} for $\FFF$ and $\GGG$). Let $\FFF_l:\ZHLC\lra \LMBool$ and $\GGG_l:\LMBool\lra \ZHLC$ be the restrictions of the contravariant functors $\FFF$ and $\GGG$, respectively.  Then, obviously,  the conravariant functors $\FFF_0^{\, l}$ and $\FFF_l$, as well as $\GGG_0^{\, l}$ and $\GGG_l$, are naturally isomorphic. Thus, the contravariant functors $\FFF_l$ and $\GGG_l$ realize the duality between the categories\/ $\ZHLC$ and\/ $\LMBool$. The formulas with which  the conravariant functors $\FFF_l$ and $\GGG_l$ are described coincide with those describing the conravariant functors $\FFF$ and $\GGG$; they are given in the proof of Theorem \ref{nzduality}.
\end{proof}

\end{document}